\documentclass[12pt]{article}
\usepackage[margin=1in]{geometry}
\usepackage{amsmath,amssymb,amsthm}
\usepackage{booktabs}
\usepackage{enumitem}
\usepackage{graphicx}
\usepackage{hyperref}

\newtheorem{theorem}{Theorem}
\newtheorem{proposition}[theorem]{Proposition}
\newtheorem{corollary}[theorem]{Corollary}

\theoremstyle{definition}
\newtheorem{definition}[theorem]{Definition}

\theoremstyle{remark}
\newtheorem{remark}[theorem]{Remark}

\title{Table-Based Encodings for Conway's Doomsday Algorithm:\\
Vectorized Doomsdays and Doomyears}
\author{Thomas Wollin\\
\small Keswick, Ontario, Canada}
\date{March 28, 2026}

\begin{document}
\maketitle

\begin{abstract}
Conway's Doomsday Algorithm (1973) determines the day of the week for
any date in the Gregorian calendar via three additive components: a
century anchor, a year offset, and a month--day offset.  The century
anchor is a fixed four-entry table.  The other two components require
live arithmetic: the year offset demands computing
$y + \lfloor y/4 \rfloor \pmod{7}$, and the month--day offset
requires a subtraction that can produce negative intermediate values.
We present two new encoding schemes that replace both arithmetic steps
with structured table lookups.

The first, \emph{vectorized doomsdays}, re-encodes each month's
doomsday date as a two-digit number whose tens and units digits
represent the backward and forward gaps (respectively) from the
nearest multiples-of-seven month anchors.  A directional crossing
rule (the ``square knot rule'') pairs the target date's gap with the
opposite-direction digit, reducing the month--day offset to a
single-digit addition.

The second, \emph{Doomyears}, encodes the year-offset function as a
navigational lookup exploiting the 28-year periodicity of the
Gregorian weekday cycle.

Together with Conway's century anchor table, these form a unified
system we call the \emph{Calamity Tables}.  We prove correctness,
establish self-verification properties, analyse the internal
structure of both encodings, and compare the cognitive complexity of
the Calamity Table system against the standard arithmetic method.
\end{abstract}

\section{Introduction}

The Doomsday Algorithm, introduced by Conway~\cite{conway1973} and
described in~\cite{berlekamp1982}, computes the day of the week for
any date as:
\begin{equation}\label{eq:main}
  \text{day of week} \equiv c + \omega(y) + \delta(m, d) \pmod{7},
\end{equation}
where $c$ is the century anchor, $\omega(y) = (y + \lfloor y/4
\rfloor) \bmod 7$ is the year offset, and $\delta(m, d)$ is the
signed distance from the target date to the month's doomsday date.

Step~$c$ is a four-entry table.  Steps~$\omega$ and~$\delta$ require
real-time arithmetic.  We replace both with pre-computed encodings.
The three tables --- century anchors, vectorized doomsdays, and
Doomyears --- we collectively term the \emph{Calamity Tables}, in
keeping with the eschatological naming tradition of the algorithm
they extend.

\section{Part 1: Vectorized doomsdays}

\subsection{Month anchors and gap structure}

\begin{definition}
The \emph{month anchors} are $A = \{0, 7, 14, 21, 28\}$, the
non-negative multiples of~$7$ up to~$28$.  The values $7, 14, 21, 28$
occur as valid calendar dates in every month; $0$ is the natural lower
boundary of the date range and serves as the anchor for dates
$1$--$6$.
\end{definition}

\begin{definition}
For a date~$t$ with $1 \leq t \leq 31$, let $a^-(t) = \max\{a \in A
: a \leq t\}$ and $a^+(t) = \min\{a \in A : a \geq t\}$.  Define the
\emph{forward gap} $g_F(t) = t - a^-(t)$ and the \emph{backward gap}
$g_B(t) = a^+(t) - t$.  When $t$ falls on a month anchor,
$g_F(t) = g_B(t) = 0$; otherwise $g_F(t) + g_B(t) = 7$.  Since
$0 \in A$, dates $1$--$6$ have $a^-(t) = 0$ and $a^+(t) = 7$, giving
$g_F(t) = t$ and $g_B(t) = 7 - t$ --- both well-defined, and either
may be used in the square knot rule.  Dates $29$--$31$ have
$a^-(t) = 28$; the next multiple of~$7$ above is~$35$, which lies
outside the calendar range and is discussed in
Remark~\ref{rem:anchor35} below.
\end{definition}

\begin{proposition}\label{prop:gapsum}
For any date $t$ with $1 \leq t \leq 28$ not on an anchor,
$g_F(t) + g_B(t) = 7$.  When $t \in A$, $g_F(t) = g_B(t) = 0$.
\end{proposition}

\begin{proof}
If $t$ lies strictly between consecutive anchors $a^-$ and $a^+$,
then $a^+ - a^- = 7$, so $g_F + g_B = (t - a^-) + (a^+ - t) = 7$.
This holds for all consecutive pairs in $A = \{0,7,14,21,28\}$,
including the pair $(0, 7)$ covering dates $1$--$6$.
If $t \in A$, then $a^- = a^+ = t$.
\end{proof}

\subsection{Construction of the vectorized doomsdays}

\begin{definition}
Let $d_m$ denote the traditional doomsday date for month~$m$.  The
\emph{vectorized doomsday} for month~$m$ is
\[
  V_m = 10 \cdot g_B(d_m) + g_F(d_m).
\]
\end{definition}

The explicit values are:

\begin{table}[h]
\centering
\begin{tabular}{@{}lcccccccccccc@{}}
\toprule
Month & Jan & Feb & Mar & Apr & May & Jun & Jul & Aug & Sep & Oct & Nov & Dec\\
\midrule
$d_m$ & 3 & 28 & 7 & 4 & 9 & 6 & 11 & 8 & 5 & 10 & 7 & 12\\
$V_m$ & 43 & 00 & 00 & 34 & 52 & 16 & 34 & 61 & 25 & 43 & 00 & 25\\
\bottomrule
\end{tabular}
\caption{Traditional doomsday dates and vectorized doomsdays.
Leap year: Jan~$d_m = 4 \to V = 34$; Feb~$d_m = 29 \to V = 61$.}
\label{tab:vdoom}
\end{table}

\begin{proposition}[Digit-sum property]\label{prop:digitsum}
For every vectorized doomsday $V_m$, either both digits are~$0$ or
the digits sum to~$7$.
\end{proposition}

\begin{proof}
Immediate from Proposition~\ref{prop:gapsum}: the tens digit is
$g_B(d_m)$ and the units digit is $g_F(d_m)$.
\end{proof}

\begin{proposition}[Spatial encoding]
The tens (left) digit of $V_m$ encodes the leftward (backward) gap;
the units (right) digit encodes the rightward (forward) gap.  The
direction a digit represents matches the direction one reads to find
it.
\end{proposition}

\subsection{The square knot rule}

\begin{theorem}[Correctness of the month--day offset]\label{thm:month}
Let $t$ be a target date in month~$m$, with forward gap $g_F(t)$ from
its nearest lower anchor.  Let $V_m = 10b + f$ where $b = g_B(d_m)$
and $f = g_F(d_m)$.  Then
\[
  t - d_m \equiv g_F(t) + b \pmod{7}.
\]
Similarly, if the backward gap $g_B(t)$ is used instead, then
\[
  d_m - t \equiv g_B(t) + f \pmod{7}.
\]
\end{theorem}

\begin{proof}
Write $t = a^-(t) + g_F(t)$ where $a^-(t)$ is a multiple of~$7$.
Write $d_m = a^+(d_m) - g_B(d_m) = a^+(d_m) - b$ where $a^+(d_m)$
is a multiple of~$7$.
Then $t - d_m = a^-(t) + g_F(t) - a^+(d_m) + b \equiv g_F(t) + b
\pmod{7}$, since both $a^-(t)$ and $a^+(d_m)$ are multiples of~$7$.

The second identity follows by symmetry: $d_m - t \equiv g_B(t) + f
\pmod{7}$ by the same argument with roles exchanged.
\end{proof}

\begin{remark}
The forward-gap case gives a result in the context of ``days after
doomsday'' (added to the year's doomsday), while the backward-gap
case gives ``days before doomsday'' (subtracted).  The choice of
direction at the target date determines both which digit is selected
\emph{and} the sign of the result --- these are not independent
decisions.
\end{remark}

\begin{remark}[Why same-direction pairing requires subtraction]
One might ask whether the encoding could instead pair same directions
--- forward with forward, backward with backward --- eliminating the
crossing rule.  Same-direction pairing is valid, but only under
\emph{subtraction}: $g_F(t) - g_F(d_m) \equiv t - d_m \pmod{7}$,
since both reference anchors are multiples of~$7$ and cancel.
Under \emph{addition}, $g_F(t) + g_F(d_m) \equiv t + d_m \pmod{7}$,
which differs from the correct value $t - d_m$ by $2d_m$.  One could
reverse the digit positions in $V_m$ to support same-direction
subtraction, but this reintroduces the negative intermediate values
that the encoding was designed to eliminate.  The cross-direction
rule is therefore the structural cost of an all-addition,
all-non-negative system.
\end{remark}

\begin{remark}[Anchor~$35$]\label{rem:anchor35}
The sequence of multiples of~$7$ continues beyond the calendar: $35$
is the natural upper boundary anchor for dates $29$--$31$, giving
backward gaps of $6, 5, 4$ respectively.  Together,
$\{0, 7, 14, 21, 28, 35\}$ form the complete set of multiples of~$7$
that bracket the full date range $[1, 31]$, and the gap
complementarity $g_F + g_B = 7$ extends uniformly across all
inter-anchor intervals, including $(28, 35)$.  In practice,
anchor~$35$ is rarely invoked: dates $29$--$31$ lie at most $3$ steps
forward from anchor~$28$, making the forward gap small and the
forward direction the natural choice.  Anchor~$0$, by contrast,
adds genuine operational value --- it gives dates $1$--$6$ a natural
lower anchor, brings them under the same uniform rule as all other
inter-anchor dates, and removes their previous treatment as a
boundary exception.
\end{remark}

\begin{corollary}[Self-consistency]\label{cor:selfcheck}
For every month~$m$, running the traditional doomsday date $d_m$
through its own vectorized doomsday yields offset~$0$.
\end{corollary}

\begin{proof}
$g_F(d_m) + g_B(d_m) = 7 \equiv 0 \pmod{7}$ (or $0 + 0 = 0$ when
$d_m$ is an anchor).
\end{proof}

\subsection{Structure of the code space}

\begin{proposition}
Exactly four distinct non-zero digit pairs appear among the vectorized
doomsdays: $\{1,6\}$, $\{2,5\}$, $\{3,4\}$, and $\{0,0\}$.  Each
non-zero pair appears in both orientations (e.g., $16$ and $61$).
\end{proposition}

\begin{proof}
For any date $t$ not on a month anchor, Proposition~\ref{prop:gapsum}
gives $g_F(t) + g_B(t) = 7$ with $g_F(t), g_B(t) \in \{1, 2, 3, 4,
5, 6\}$.  The distinct unordered pairs $\{g_B, g_F\}$ summing to~$7$
are exactly $\{1,6\}$, $\{2,5\}$, and $\{3,4\}$.  When $t$ is a month
anchor, $g_F(t) = g_B(t) = 0$, giving the pair $\{0,0\}$.  Thus at
most four distinct unordered pairs can arise.  Inspection of
Table~\ref{tab:vdoom} confirms that each pair is realised by at least
one traditional doomsday date, and each non-zero pair appears in both
orientations.
\end{proof}

\subsection{Equivalence classes of anchor-date systems}

The vectorized doomsday construction does not depend on Conway's
specific choice of anchor dates.  Any set of twelve dates (one per
month) that share a common weekday within each year can serve as
input, and the construction produces a valid code table.  We now
classify all such systems.

\begin{definition}
Two anchor-date systems $\{d_m\}$ and $\{d'_m\}$ are
\emph{equivalent} if $d_m \equiv d'_m \pmod{7}$ for every month~$m$.
\end{definition}

Since the vectorized doomsday $V_m$ depends only on $d_m \bmod 7$,
equivalent systems produce identical code tables.

\begin{proposition}[Seven equivalence classes]\label{prop:seven}
There are exactly~$7$ equivalence classes of same-weekday
anchor-date systems, indexed by a uniform offset
$k \in \{0,1,\ldots,6\}$ from Conway's doomsday dates.  System~$k$
has anchor-date residues $d_m + k \pmod{7}$ and century anchors
$c - k \pmod{7}$, where $c$ is Conway's century anchor.
\end{proposition}

\begin{proof}
Let $\{d_m\}$ be Conway's doomsday dates and $\{d'_m\}$ any other
same-weekday system.  In a given year, all $d_m$ share a weekday~$W$
and all $d'_m$ share a weekday~$W'$.  Within each month~$m$, the
weekday of the $n$th is determined by $n \bmod 7$ (up to a
month-specific constant), so $d'_m \equiv d_m + (W' - W) \pmod{7}$
for every~$m$.  Setting $k = W' - W \bmod 7$ gives a uniform shift.
Since the overall algorithm must yield the same final weekday, the
century anchor absorbs the complementary shift: $c' = c - k \pmod{7}$.
\end{proof}

\begin{proposition}[Invariance of month groupings and code vocabulary]
All seven systems produce the same set of codes
$\{00, 16, 25, 34, 43, 52, 61\}$ and preserve the same month
groupings (which months share a code).  Incrementing~$k$ by~$1$
cycles every code one step through the rotation
$00 \to 61 \to 52 \to 43 \to 34 \to 25 \to 16 \to 00$.
\end{proposition}

\begin{proof}
The code for month~$m$ in system~$k$ depends only on
$(d_m + k) \bmod 7$.  Incrementing~$k$ uniformly shifts every
month's residue, applying the same cyclic permutation of codes
to all months simultaneously.  Since months with
$d_m \equiv d_{m'} \pmod{7}$ undergo the same shift, they always
receive the same code.
\end{proof}

The seven month groupings, determined by the residue classes of
Conway's doomsday dates, are:

\medskip
\centerline{%
\begin{tabular}{@{}clc@{}}
\toprule
$d_m \bmod 7$ & Months & Size\\
\midrule
0 & Feb, Mar, Nov & 3\\
1 & Aug & 1\\
2 & May & 1\\
3 & Jan, Oct & 2\\
4 & Apr, Jul & 2\\
5 & Sep, Dec & 2\\
6 & Jun & 1\\
\bottomrule
\end{tabular}}
\medskip

\begin{proposition}[Optimality of Conway's system]\label{prop:optimal}
Among the seven equivalence classes, Conway's ($k = 0$) is the unique
class that places the zero code~$00$ on the largest month group
(February, March, November --- three months).  Every other class
places~$00$ on a group of size at most~$2$.
\end{proposition}

\begin{proof}
The residue-$0$ group $\{\text{Feb, Mar, Nov}\}$ has size~$3$; all
other residue classes have size $\leq 2$.  Only $k = 0$ maps
residue~$0$ to code~$00$.
\end{proof}

\begin{remark}
Each equivalence class admits valid anchor dates --- Wang's null-days
($k = 5$) demonstrate this with dates such as Feb~12 and Nov~12.
The optimality of $k = 0$ is not a question of feasibility but of
\emph{efficiency}: three zero-months versus at most two for any
other class.  Code~$00$ is the most efficient to use (the month
contributes nothing to the sum, so the month--day step reduces to a
single gap measurement), making the number of zero-months a natural
measure of system quality.
\end{remark}

\begin{remark}
Wang's null-days algorithm~\cite{wang2014} uses anchor dates
$1/1$, $3/5$, $5/7$, $7/9$, $9/3$, $2/12$, $12/10$, $10/8$,
$8/6$, $6/4$, $4/2$, $11/12$,
which belong to the $k = 5$ equivalence class: every Wang date
satisfies $d'_m \equiv d_m + 5 \pmod{7}$.  The resulting vectorized
code table is $61, 25, 25, 52, 00, 34, 52, 16, 43, 61, 25, 43$
--- a rotation of Conway's with only one zero-month (May) instead
of three, and century anchors shifted by $-5$ (e.g.\ the 2000s
anchor becomes $(2 - 5) \bmod 7 = 4$).
\end{remark}

\section{Part 2: Doomyears}

\subsection{Periodicity}

\begin{theorem}[28-year periodicity]\label{thm:period}
$\omega(y + 28) = \omega(y)$ for all $y \geq 0$.
\end{theorem}

\begin{proof}
$\lfloor(y+28)/4\rfloor = \lfloor y/4\rfloor + 7$, so the total
increases by $28 + 7 = 35 \equiv 0 \pmod{7}$.
\end{proof}

\begin{corollary}
The anchor years $0, 28, 56, 84$ satisfy $\omega(y) = 0$.
\end{corollary}

\subsection{Packed-number sequences}

\begin{definition}
For $d \in \{0, \ldots, 15\}$, define $f(d) = \omega(d)$,\;
$b(d) = \omega(28 - d)$,\; and the packed numbers
$F_d = 10d + f(d)$,\; $B_d = 10d + b(d)$.
\end{definition}

\begin{definition}
The \emph{Doomyear} at distance~$d$ is $D_d = 100d + 10 \cdot b(d) + f(d)$,
encoding both directions in a single number: distance, backward digit,
forward digit.
\end{definition}

\begin{table}[h]
\centering
\resizebox{\linewidth}{!}{%
\begin{tabular}{@{}r@{\;\;}cccccccccccccccc@{}}
\toprule
$d$ & 0 & 1 & 2 & 3 & 4 & 5 & 6 & 7 & 8 & 9 & 10 & 11 & 12 & 13 & 14 & 15\\
\midrule
$F_d$ & 00 & 11 & 22 & 33 & 45 & 56 & 60 & 71 & 83 & 94 & 105 & 116 & 121 & 132 & 143 & 154\\
$B_d$ & 00 & 15 & 24 & 33 & 42 & 50 & 66 & 75 & 84 & 92 & 101 & 110 & 126 & 134 & 143 & 152\\
$D_d$ & 00 & 151 & 242 & 333 & 425 & 506 & 660 & 751 & 843 & 924 & 1015 & 1106 & 1261 & 1342 & 1433 & 1524\\
\bottomrule
\end{tabular}}
\caption{Forward, backward, and interleaved Doomyear sequences.}
\end{table}

\subsection{Structural properties}

\begin{proposition}[Leap-year stumble]
$\Delta_f(d) = f(d) - f(d-1) \pmod{7}$ equals~$2$ when $4 \mid d$
and~$1$ otherwise.
\end{proposition}

\begin{proof}
$\Delta_f(d) = 1 + \lfloor d/4\rfloor - \lfloor(d-1)/4\rfloor$.
The floor difference is~$1$ when $4 \mid d$, else~$0$.
\end{proof}

\begin{proposition}[Complementarity]\label{prop:complement}
For all $d \in \{0, \ldots, 15\}$,
\[
  f(d) + b(d) \equiv
  \begin{cases}
    0 \pmod{7} & \text{if } 4 \mid d, \\
    6 \pmod{7} & \text{otherwise}.
  \end{cases}
\]
\end{proposition}

\begin{proof}
$f(d) + b(d) = d + \lfloor d/4\rfloor + (28-d) +
\lfloor(28-d)/4\rfloor \pmod{7} = 28 + \lfloor d/4\rfloor +
\lfloor(28-d)/4\rfloor \pmod{7}$.
Writing $d = 4q + r$: when $r = 0$,
$\lfloor d/4\rfloor + \lfloor(28-d)/4\rfloor = q + (7-q) = 7
\equiv 0$; when $r > 0$,
$= q + (6-q) = 6$.  Since $28 \equiv 0$, the result is~$0$ or~$6$.
\end{proof}

\begin{remark}
Proposition~\ref{prop:complement} shows that the Doomyear digit-sum
property is \emph{not} identical to that of the vectorized doomsdays.
The month codes always have digits summing to~$7$ (or both~$0$), a
consequence of the fixed seven-unit span between consecutive anchors.
The Doomyear directional digits sum to~$0 \pmod 7$ (i.e.\ both~$0$
or summing to~$7$) only at multiples-of-$4$ distances; at all other
distances they sum to~$6 \pmod 7$.  The structural parallel between
the two encodings holds in form --- directional digits packed by
position --- but the verification guarantee is weaker for Doomyears
and should not be overstated.
\end{remark}

\section{Cognitive complexity comparison}

\begin{table}[h]
\centering
\begin{tabular}{@{}lcc@{}}
\toprule
\textbf{Operation} & \textbf{Standard} & \textbf{Calamity Tables}\\
\midrule
\multicolumn{3}{@{}l@{}}{\emph{Year-offset step}}\\
\quad Integer division ($\lfloor y/4\rfloor$) & 1 & 0\\
\quad Multi-digit addition & 1 & 0\\
\quad Modular reduction (mod~7, large) & 1 & 0\\
\quad Small subtraction (distance to anchor) & 0 & 1\\
\quad Table recall & 0 & 1\\
\midrule
\multicolumn{3}{@{}l@{}}{\emph{Month--day step}}\\
\quad Subtraction (target $-$ doomsday) & 1 & 0\\
\quad Sign correction / mod~7 & 1 & 0\\
\quad Gap measurement (to month anchor) & 0 & 1\\
\quad Digit selection + single-digit add & 0 & 1\\
\midrule
\textbf{Total operations} & 5 & 4\\
\textbf{Operation dependency} & Serial & Independent\\
\textbf{Max intermediate value} & 124 & 6\\
\textbf{Divisions required} & 1 & 0\\
\textbf{Mod reductions on large numbers} & 2 & 0\\
\bottomrule
\end{tabular}
\caption{Mental operations for the two variable steps combined.}
\end{table}

The operation count understates the difference.  The standard method's
operations are serial (each depends on the prior result) and involve
intermediate values up to~$124$.  The Calamity Table operations
involve no value larger than~$6$ and the two steps (month and year)
are fully independent.

\section{Memorisation burden}

The complete Calamity Table system requires:
\begin{itemize}[nosep]
  \item 4 century anchors (Conway's original).
  \item 12 vectorized doomsday codes (4 distinct non-zero pairs in
        two orientations, plus three zero-months).
  \item 16 Doomyears (or equivalently, $2 \times 16$ directional
        packed numbers).
\end{itemize}

Structural redundancy reduces the effective burden: month code digits
sum to~$7$; mirror pairs ($16/61$, $25/52$, $34/43$) halve the
distinct codes; the leap-year stumble rhythm makes the year sequences
reconstructable from checkpoints; and Doomyear directional digits sum
to~$7$ at multiples-of-$4$ distances (and to~$6$ otherwise),
providing a partial derivation check.

\section{Historical context}

The 28-year periodicity is classical (Carroll~\cite{carroll1887}).
Various year-offset shortcuts exist, including Conway's ``odd$+$11''
rule~\cite{berlekamp1982}.  The month--day step has historically been
handled by memorising the traditional doomsday dates and performing a
subtraction.

The vectorized doomsday encoding --- representing each month's
doomsday information as a directionally-aware, self-verifying
two-digit code derived from multiple-of-$7$ anchors --- appears to be
new.  The structural parallel between the month codes and the
Doomyear codes (both using complementary digit pairs, spatial
direction encoding, and digit-sum verification) unifies the two
variable steps of the algorithm under a single design principle ---
the Calamity Tables.

The classification of anchor-date systems into seven equivalence
classes (Section~2.5) and the proof that Conway's class is uniquely
optimal for the vectorized encoding also appear to be new.  Wang's
null-days~\cite{wang2014}, which belong to a different equivalence
class ($k = 5$), confirm that the encoding technique generalises
beyond Conway's specific dates, while the optimality result
(Proposition~\ref{prop:optimal}) shows that Conway's dates yield
the best code table among all seven classes.

\end{document}